\documentclass[12pt,twoside]{article}
\usepackage[english]{babel}
\usepackage[latin1]{inputenc}
\usepackage[T1]{fontenc}

\usepackage[pdftex]{graphicx}
\usepackage{graphics}

\usepackage{amsmath}
\usepackage{amssymb}
\usepackage{amsthm}

\usepackage{array}
\usepackage{color}
\usepackage{hyperref}

\theoremstyle{plain}
\newtheorem{theorem}{Theorem}[section]
\newtheorem{lemma}[theorem]{Lemma}
\newtheorem{prop}[theorem]{Proposition}

\theoremstyle{definition}
\newtheorem{definition}{Definition}[section]

\theoremstyle{remark}
\newtheorem*{remark}{Remark}

\newcommand{\R}{\mathbb{R}}

\newcommand{\Pro}{\mathcal{P}}
\newcommand{\pr}{\mathbb{P}}

\title{Distribution's template estimate with Wasserstein metrics} 

\author{Emmanuel Boissard, Thibaut Le Gouic, Jean-Michel Loubes}
\date{{Institut de Math\'ematiques de Toulouse, Universit\'e Toulouse Paul Sabatier }}

\begin{document}
\maketitle

\begin{abstract} 
In this paper we tackle the problem of comparing distributions  of random variables and defining a mean pattern between a sample of random events. Using barycenters of measures in the Wasserstein space, we propose an iterative version as an estimation of the mean distribution. Moreover, when the distributions are  a common measure warped by a centered random operator, then the barycenter enables to recover this distribution template. 
\end{abstract}

\vskip .1in
\indent {\bf e-mail}: boissard,tlegouic,loubes@math.univ-toulouse.fr

\section{Introduction}
Giving a sense to the  notion of {\it mean behaviour} may be counted among the very early activities of statisticians. When confronted to a large sample of high dimensional data, the usual notion of Euclidean mean is not usually enough since the information conveyed by the data possesses an inner geometry far from the Euclidean one. Indeed, deformations  on the data such as translations, scale location models for instance or more general warping procedures prevent the use of the usual methods in data analysis.  The mere issue of defining the mean of the data becomes a difficult task. This problem arises naturally for a wide range of statistical research fields such as functional data analysis for instance in \cite{Gamboa-Loubes-Maza-07}, \cite{MR2168993}, \cite{2011arXiv1101.0736B}  and references therein, image analysis in \cite{TrouveY05}  or \cite{JASA}, shape analysis in \cite{kendall} or \cite{bb206} with many applications ranging from biology in \cite{Bolstad-03} to pattern recognition \cite{Sakoe-Chiba-78} just to name a few.\vskip .1in
Without any additional knowledge, this problem is too difficult to solve. Hence to tackle this issue, two main directions have been investigated. On the one hand, some assumptions are made on  the deformations. Models governed by parameters have been proposed, involving for instance  scale location parameters, rotations, actions of parameters of Lie groups as in \cite{myptrf} or in a more general way deformations parametrized by their coefficients on a given basis or in an RKHS set \cite{allason}. Adding structure on the deformations enables to define the {\it mean behaviour} as the data warped by the {\it mean deformation}, i.e the deformation parametrized by the mean of the parameters. Bayesian  or semi-parametric statistics enable to provide sharp estimation of these parameters. However, the consistency of the estimator remains  a theoretical issue for many cases.\\
\indent On the other hand, another direction consists in finding an adequate distance between the data which reveals the information which is conveyed.  Actually, the chosen distance depends on the nature of the set where the observations belong, whose estimation is a hard task. We refer for instance to  \cite{Pennec2006} for some examples. Once an appropriate distance has been chosen, difficulties arise when trying to define the mean as the minimum of the square distance since both existence and uniqueness rely on assumptions on the geometry of the data sets. This will be the framework of our work. 
\vskip .1in Assume that we observe $j=1,\dots,J$ samples of $i=1,\dots,n$ independent random variables $X_{i,j}\in \mathbb{R}^d$ with distribution $\mu_j$. We aim at defining the {\it mean} behaviour of these observations, i.e their {\it mean} distribution. For this we will extend  the notion of barycenter of the distributions with respect to  the Wasserstein distance defined  in \cite{ agueh2010barycenters} to the empirical measures and prove the consistency of its estimate. Actually, Wasserstein distance is a powerful tool to compute distance between distributions, with application in statistics pioneered in \cite{MR1698999}, \cite{MR2435470} or \cite{MR1740113} for instance.
 Moreover, we will tackle the case where the distributions are the images of  an unknown original distribution by  random operators under some  suitable assumptions. In this case, we prove that an iterative version of the barycenter of the empirical distributions provides an estimate which enables to recover the original template distribution when the number of replications $J$ is large enough. \vskip .1in
The paper falls into the following parts. Section~\ref{s:bary} is devoted to the extension of the notion of Barycenter in the Wasserstein space for empirical measures. In Section~\ref{s:iterate}, we consider a modification of the notion of barycenter by considering iterative barycenters, which have the advantage to enable to recover the distribution pattern as proved in Section~\ref{s:template}. Finally, some  data applications are outlined in Section~\ref{s:appli}.

\section{Barycenters in the Wasserstein space: Notations and general results} \label{s:bary}

Let $(E, d, \Omega)$ denote a metric measurable space. The set of probability measures
over $E$ is denoted by $\Pro(E)$.
Given a collection of probability measures $\mu_ 1, \ldots, \mu_J$ over $E$,
and weights $\lambda_1, \ldots, \lambda_J \in \R$, $\lambda_j \geq 0$, 
$1 \leq j \leq J$, $\sum_{j = 1}^J \lambda_j = 1$, 
there are several 
natural ways to define a weighted average of these measures.
Perhaps the most straightforward is to take the convex combination of these measures

\begin{equation*}
 \mu_c = \sum_{j = 1}^J \lambda_j \mu_j,
\end{equation*}
using the fact that probability measures form a convex subset of the linear space of finite measures.
However, if we provide $\Pro(E)$ with some metric structure, 
the definition above is not really appropriate.

We denote by $\Pro_2(E)$ the set of all probability measures over $E$ with a finite second-order moment.
Given two measures $\mu$, $\nu$ in $\Pro(E)$, we denote by $\Pro(\mu, \nu)$ the set of all
probability measures $\pi$ over the product set $E \times E$ 
with first, resp. second, marginal $\mu$, resp. $\nu$.

The transportation cost with quadratic cost function, or quadratic transportation cost, 
between two measures
$\mu$, $\nu$  in
$\Pro_2(E)$, is defined as

\begin{equation*}
 \mathcal{T}_2(\mu, \nu) = \inf_{ \pi _\in \Pro(\mu, \nu)} \int d(x, y)^2 d \pi.
\end{equation*}

The quadratic transportation cost allows to endow the set of probability measures 
(with finite second-order moment)
 with a metric by setting

\begin{equation*}
 W_2(\mu, \nu) = \mathcal{T}_2(\mu, \nu)^{1/2}.
\end{equation*}

This metric is known under the name of $2$-Wasserstein distance.

In Euclidean space, the barycenter of the points $x_1, \ldots, x_J$ 
with weights $\lambda_1, \ldots, \lambda_J$,
$\lambda_j \geq 0$, $\sum_{j = 1}^J \lambda_j= 1$, is defined as

\begin{equation*}
 b = \sum_{j = 1}^J \lambda_j x_j.
\end{equation*}

It is also the unique minimizer of the functional

\begin{equation*}
y\mapsto E(y) = \sum_{j = 1}^J \lambda_j |x_j - y|^2.
\end{equation*}

By analogy with the Euclidean case, we give the following definition for Wasserstein barycenter,
 introduced 
by M.~Agueh and G.~Carlier in 
\cite{agueh2010barycenters}.

\begin{definition}

 We say that the measure $\mu \in \Pro_2(E)$ is a Wasserstein barycenter for 
the measures $\mu_1, \ldots, \mu_J \in \Pro_2(E)$ endowed with
weights $\lambda_1, \ldots, \lambda_J$, where $\lambda_j \geq 0$, $ \leq j \leq J$, 
and $\sum_{j = 1}^J \lambda_j = 1$,
if $\mu$ minimizes

\begin{equation*}
 E(\nu) = \sum_{j = 1}^J \lambda_j W_2^2(\nu, \mu_j).
\end{equation*}

We will write

\begin{equation*}
 \mu_B(\lambda) = \text{Bar}( (\mu_j, \lambda_j)_{1 \leq j \leq J}).
\end{equation*}

\end{definition}

In other words, the barycenter is the weighted Fr\'echet mean in the Wasserstein space.
In \cite{agueh2010barycenters}, the authors prove that when
$E = \R^d$  the barycenter exists. They also provide suitable assumptions on the measures $\mu_j$, $1 \leq j \leq J$ to ensure that the barycenter  is unique. For example, a sufficient condition is that
one of the measures $\mu_j$ admits a density with respect to the Lebesgue measure.
They also provide a problem that is the dual of the minimization of the
functional $E$ defined above, as well as characterizations of the barycenter. \vskip .1in

Next, we recall a version of Brenier's theorem on the characterization of quadratic optimal 
transport in $\R^d$. Throughout all the paper we will use the following notation. 

\begin{definition}
 Let $E$, $F$ be measurable spaces and $\mu \in \Pro(E)$. Let $T : E \rightarrow F$ be a measurable map.
The push-forward of $\mu$ by $T$ is the probability measure $T_\# \mu \in \Pro(F)$ defined by
the relations

\begin{equation*}
 T_\# \mu(A) = \mu(T^{-1}(A)), \quad A \subset F \text{ measurable.}
\end{equation*}

\end{definition}
Hence Brenier's theorem can be stated as follows.

\begin{theorem}[Brenier's theorem, see \cite{brenier1991polar}]
{Let $\mu, \nu \in \Pro_2(\R^d)$ be  measures}, with $\mu$ 
absolutely continuous w.r.t. Lebesgue measure. Then there exists a $\mu$-a.e. unique map 
$T : \R^d \rightarrow \R^d$ such that

\begin{itemize}
 \item $T_\# \mu = \nu$,
 \item $W_2^2 (\mu, \nu) = \int_{\R^d} |T(x) - x|^2 \mu(dx)$.
\end{itemize}

Moreover, there exists a lower semi-continuous convex function $\varphi : \R^d \rightarrow \R$ such that
$T = \nabla \varphi$ $\mu$-a.e., and $T$ is the only map of this type pushing forward $\mu$ to $\nu$, up to a
$\mu$-negligible modification. The map $T$ is called the Brenier map from $\mu$ to $\nu$.

\end{theorem}

\begin{remark}

The theorem above is commonly referred to as Brenier's theorem and originated from Y. Brenier's work in the analysis and mechanics literature.
Much of the current interest in transportation problems emanates from this area of mathematics. We conform to the common use of the name.
However, it is worthwile pointing out that a similar statement was established earlier independently in a
probabilistic framework by J.A. Cuesta-Albertos and C. Matr\'{a}n \cite{cuesta1989notes} : they show existence of an optimal transport map for quadratic cost over
Euclidean and Hilbert spaces, and prove monotonicity of the optimal map in some sense (Zarantarello monotonicity).

\end{remark}

As observed in \cite{agueh2010barycenters}, the barycenter of two measures is the interpolant of these two 
measures in the sense of McCann.

\begin{prop}[See \cite{agueh2010barycenters}, Section 6.2]
 Let $\mu, \nu \in \Pro_2(\R^d)$ be absolutely continuous w.r.t. Lebesgue measure. Let $T : \R^d \rightarrow \R^d$
denote the Brenier map from $\mu$ to $\nu$. The barycenter of $(\mu, \lambda)$ and $(\nu, 1 - \lambda)$
is

\begin{equation*}
 \mu_\lambda = \left( \lambda  \text{Id} + (1 - \lambda) T \right)_\# \mu.
\end{equation*}

\end{prop}
This provides a natural expression for the barycenter of measures as a convex combination of measures. 

\section{Estimation of Barycenters of empirical measures}
 Assume we do not observe the distributions $\mu_j$'s but approximations of these distributions. Let $\mu_j^n \in \Pro_2(\R^d)$ for $1 \leq j \leq J$ be these approximations  in the sense that they converge with respect to Wasserstein distance, i.e $W_2( \mu^n_j, \mu_j) \rightarrow 0$ when $n\rightarrow + \infty$. Our aim is to study the asymptotic behaviour of the barycenter of the $\mu^n_j$'s when $n$ goes to infinity.
\subsection{Consistency of the approximated barycenter}
We are interested here in statistical properties of the barycenter of the $\mu_1^n,\dots,\mu_J^n$. We begin by establishing a consistency result.

\begin{theorem} \label{th:EmpiricVersion}

Let $J \geq 1$, and for every $n \geq 0$, let $\mu_j^n \in \Pro_2(\R^d)$, $1 \leq j \leq J$,
be probability measures converging in Wasserstein topology to some probability measure $\mu_j$ for $1 \leq j \leq J$. 
Let $\lambda_1, \ldots, \lambda_J$ be positive weights.
Let $\mu^n_B$ be a barycenter of the $(\mu^n_j,\lambda_j)$.
The sequence $(\mu^n_B)_{n\geq 1}$ is compact and any of its limit points lies in $\text{Bar}( (\mu_j, \lambda_j)_{1 \leq j \leq J})$.
\end{theorem}

Note that If any of the $\mu^n_j$ is absolutely continuous with respect to the Lebesgue measure, then $\mu_B^n$ is unique and our theorem states that it converges to a barycenter of the limit measures $(\mu_j,\lambda_j)$. Likewise, if any of the $\mu_j$ is absolutely continuous with respect to the Lebesgue measure, any $\mu_B^n$ is converging to the unique barycenter of $(\mu_j,\lambda_j)$. \vskip .1in

The proof of this theorem relies on the following lemma which provides a characterization of a barycenter of measures. The  proof of the lemma is inspired by the proof of   Proposition 4.2 in \cite{agueh2010barycenters} and is postponed to the Appendix.

\begin{lemma}\label{lem:iff}
 Let $\Gamma(\mu_1,\ldots,\mu_J)$ be the set of probability measures on $(\mathbb{R}^d)^J$ with marginals $\mu_1,\ldots,\mu_J$, respectively and $T(x_1,\ldots,x_J)=\sum_{j=1}^J \lambda_j x_j$ with weights $\lambda_j \geq 0$ such that $\sum_{j=1}^J {\lambda_j}=1$. A probability measure $\nu$ is a barycenter of $\mu_1,...,\mu_J$  with weights $(\lambda_j)_{ \leq j \leq J}$ if and only if $\nu=T_\# \gamma$ where $\gamma \in \Gamma(\mu_1,...,\mu_J)$ minimizes
\begin{equation}\label{eq:a_minimiser}
\int \sum_{1 \leq j \leq J}\lambda_j \|T(x_1,...,x_J) - x_j\|^2 d\gamma(x_1,...,x_J).
\end{equation}
\end{lemma}

\subsection{An Iterative version of barycenters of measures} \label{s:iterate}
Barycenters in Euclidean spaces enjoy the \emph{associativity property} : 
the barycenter of $x_1, x_2, x_3$ with weights $\lambda_1, \lambda_2, \lambda_3$
coincides with the barycenter of $x_{1 2}, x_3$ with weights $\lambda_1 + \lambda_2, \lambda_3$ 
when $x_{12}$ is the barycenter of 
$x_1, x_2$ with weights $\lambda_1, \lambda_2$. 
This property, as we will see, no longer holds when considering barycenters in
Wasserstein spaces over Euclidean spaces, with the notable exception of dimension $1$. \vskip .1in
Therefore we introduce a notion of \emph{iterated barycenter}
as the point obtained by successively taking two-measures barycenters 
with appropriate weights. This does not in general coincide with the 
ordinary barycenter.
However, we will identify cases where the two notions match.

\begin{definition} \label{def:IB}
 Let $\mu_i \in \Pro_2 ( E)$, $1 \leq i \leq n$, and $\lambda_i > 0$, $1 \leq i \leq n$ with $\sum_{i = 1}^n \lambda_i = 1$.
The iterated barycenter of the measures $\mu_1, \ldots, \mu_n$ with weights $\lambda_1, \ldots, \lambda_n$
is denoted by $\text{IB}( (\mu_i, \lambda_i)_{1 \leq i \leq n})$ and is defined as follows :

\begin{itemize}
 \item $\text{IB} ( (\mu_1, \lambda_1)) = \mu_1$, 
 \item $\text{IB} ( (\mu_i, \lambda_i)_{ 1 \leq i \leq n}) = \text{Bar } \left[ ( \text{IB} ( (\mu_i, \lambda_i)_{ 1 \leq i \leq n - 1}), \lambda_1 + \ldots + \lambda_{n - 1} ), (\mu_n, \lambda_n) \right]$
\end{itemize}

\end{definition}

%

The next proposition establishes consistency of iterated barycenters of approximated measures $\mu_j^n$, for $j=1,\dots,J$.

\begin{theorem} \label{prop:ngrand}
 The iterated barycenter is consistent : if $\mu_j^n \rightarrow \mu_j$ in $W_2$ distance for $j = 1, \ldots, J$, then 

\begin{equation*}
 IB ( ( \mu_j^n, \lambda_j)_{1 \leq j \leq J}) \rightarrow IB ( (\mu_j, \lambda_j)_{1 \leq j \leq J})
\end{equation*}
 
in $W_2$ distance. 
\end{theorem}

\begin{remark}
 Iterated barycenters as well as barycenters are well-suited to computations, since
there exist efficient numerical methods to compute McCann's interpolant, 
see e.g.  \cite{benamou2000computational, haber2010efficient}. The purpose of introducing the iterative barycenters is that, as shown in the next Section~\ref{lepouquoi}, the resulting measure has an expression as the image of a measure by a linear combination of maps. This will be helpful when considering a warping setting.  Moreover, as we will see later, in some cases of interest the iterated barycenter does not depend on the order
in which two-measures barycenters are taken, allowing for parallel computation schemes.
\end{remark}

\section{Deformations of a template measure} \label{s:template}

We now would like to use Wasserstein barycenters or iterated barycenters in the following framework : let $(E, d, \Omega)$ denotes a metric measurable space and
assume that we observe probability measures in $\Pro(E)$, $\mu_1, \ldots, \mu_J$ that are deformed versions, in some sense, of
an original measure $\mu$. We would like to recover $\mu$ from the observations. 
Here, we propose to study the relevance of the barycenter as an estimator of the template measure, when
the deformed measures are of the type $\mu_j = {T_j}_\# \mu$ for suitable push-forward maps $T_j$.

Our aim here is to extend the results of J.F.~Dupuy, J.M.~Loubes and E.~Maza in \cite{dupuy2011non}.
They study the problem of \emph{curve registration}, that we can describe as follows : given an unknown 
increasing function $F : [a, b] \mapsto [0, 1]$, and a random variable $H$ with values in the set of continuous increasing
functions $h : [a, b] \mapsto [a, b]$, we observe $F \circ h_1^{-1}, \ldots, F \circ h_n^{-1}$ where $h_i$ are i.i.d.
versions of $H$ (randomly warped versions of $F$). Let $\mu \in \Pro(\R)$ denote the probability measure that admits $F$ as its c.d.f. : then the
above amounts to saying that we observe ${h_i}_\# \mu$, $1 \leq i \leq n$. The authors 
build an estimator by using quantile
functions that turns out to be the Wasserstein barycenter of the observed measures. 
They show that the estimator
converges to ${(\mathbb{E}H)}_\# \mu$.\vskip .1in
Hereafter, we first define a class of deformations for distributions, which are modeled by a push forward  action by a family of measurable maps $T_j,\: j=1,\dots,J$ undergoing the following restrictions. Such deformations will be called {\it admissible}.

\subsection{Admissible deformations}

\begin{definition}
 The set $GCF(\Omega)$ is the set of all gradients of convex functions, that is to say
 the set of all maps $T : \Omega \rightarrow \R^n$ such that there exists a proper convex l.s.c. function $\phi : \Omega \rightarrow \R$
 with $T = \nabla \phi$.
\end{definition}

\begin{definition}
 We say that the family $(T_i)_{i \in I}$ of maps on $\Omega$ is an \emph{admissible family of deformations} if the 
following requirements are satisfied :

\begin{enumerate}
 \item there exists $i_0 \in I$ with $T_{i_0} = \text{Id}$,
 \item the maps $T_i : \Omega \rightarrow \Omega$ are one-to-one and onto,
 \item for $i, j \in I$ we have $T_i \circ T_j^{-1} \in GCF(\Omega)$.
\end{enumerate}

\end{definition}

 The following Proposition provides examples are  of  such deformations.

\begin{prop} \label{prop:admi}

 The following are admissible families of deformations on domains of $\R^n$.

\begin{itemize}
 \item The set of all product continuous increasing maps on $\R^n$, i.e. the set of all maps 

\begin{equation*}
 T : x \mapsto (F_1(x_1), \ldots, F_n(x_n))
\end{equation*}

where the functions $F_i : \R \rightarrow \R$ are continuous increasing functions with $F_i \rightarrow_{- \infty} - \infty$, $F_i \rightarrow_{+ \infty} + \infty$.
 
In particular, this includes the family of scale-location transformations, i.e. maps of the type $x \mapsto a x + b$, $a > 0$, $b \in \R^n$.
\item The set of \emph{radial distorsion} transformations, i.e. the set of maps

\begin{equation*}
 T : \R^n \rightarrow \R^n, \quad x \mapsto F(|x|) \frac{x}{|x|}
\end{equation*}

where $F : \R^+ \mapsto \R^+$ is a continuous increasing function such that $F(0) = 0$.
 \item The maps $ ^t G \circ T_i \circ G$ where $(T_i)_{i \in I}$ is an admissible family of deformations on $\Omega$ and 
$G \in \mathcal{O}_n$ is a fixed orthogonal matrix. This family has $^t G( \Omega)$ as its domain.
\end{itemize}

\end{prop}
Proof of Proposition~\ref{prop:admi}
\begin{proof}

Let us consider the  first family. Checking the two first requirements is straightforward and we only take care of the last one.
Let $S : x \mapsto (F_1(x_1), \ldots, F_n(x_n))$ and $T : x \mapsto (G_1(x_1), \ldots, G_n(x_n))$. The map $S \circ T^{-1}$ is
given by

\begin{equation*}
 S \circ T^{-1}(x) = \left( F_1 \circ G_1^{-1} (x_1), \ldots, F_n \circ G_n^{-1}(x_n) \right),
\end{equation*}

and this is the gradient of the function

\begin{equation*}
 x \mapsto \int_0^{x_1} F_1 \circ G_1^{-1} (z) dz + \ldots + \int_0^{x_n} F_n \circ G_n^{-1} (z) dz.
\end{equation*}

The functions $F_i \circ G_i^{-1}$ are increasing, so that their primitives are convex functions, which makes the function above
convex.

Second point : observe that radial distortion transformations form a group, so that
we only need show that each such transformation is the gradient of a convex function. And indeed, 
$T : x \mapsto F(|x|) \frac{x}{|x|}$ is the gradient of the function 

\begin{equation*}
 x \mapsto \int_0^{|x|} F(r) d r
\end{equation*}

and this is a convex function because $F$ is increasing.

The final item is a simple consequence of the observation that if $G \in \text{GL}_n$ and
$f : \R^n \rightarrow \R$ is differentiable, then $\nabla ( f \circ G) = ^t G \circ \nabla f \circ G$.
\end{proof}

\subsection{Barycenter of measures warped using admissible deformations} \label{lepouquoi}

We are interested in  recovering
a template measure  from deformed observations. The unknown template is a probability
 measure $\mu$ on the domain $\Omega \subset \R^d$, absolutely 
continuous w.r.t. the Lebesgue measure $\lambda$.
We represent the deformed observations as push-forwards of $\mu$ by maps 
$T : \Omega \rightarrow \Omega$, i.e. we observe 
$(T_j)_\# \mu$, $j = 1, \ldots, J$.

Theorem \ref{thm_averaging_deformations_in_admissible_case} 
states that when $T_j$ belongs to an 
admissible family of deformations, taking the iterated barycenter of the 
observations corresponds to averaging the deformations. With this explicit expression at hand, we can check that in the case described above, the iterated barycenter coincides with the usual notion of barycenter.

\begin{theorem} \label{thm_averaging_deformations_in_admissible_case}
 Assume that $(T_i)_{i \in I}$ is an admissible family of deformations on a domain $\Omega \subset \R^n$, 
and let $\mu \in \Pro_2( \Omega)$, $\mu << \lambda$. Let $\mu_j = (T_j)_\# \mu$. The following holds :

\begin{equation} \label{firstpart}
 IB( ( \mu_j , \lambda_j )_{1 \leq j \leq J} ) = ( \sum_{j = 1}^J \lambda_j T_j )_\# \mu.
\end{equation}
Moreover
\begin{equation} \label{prop:clememe}
 IB( ( \mu_j , \lambda_j )_{1 \leq j \leq J} ) = \text{Bar}( ( \mu_j , \lambda_j )_{1 \leq j \leq J} ).
\end{equation}
\end{theorem}

\begin{remark}$ $ \\

\begin{enumerate}
\item The special case of the dimension 1

 In dimension $1$, the set of \emph{all} continuous increasing maps is 
an admissible family of deformations. The previous theorem 
applies for this very large class of deformations. 
Results in this case are known from \cite{dupuy2011non}  or \cite{GALLON:2011:HAL-00593476:1}: the only
new part here is that the estimator can be computed iteratively.

\item Barycenters and iterated barycenters do not match in general.

The fact that the two notions of barycenter introduced above coincide no 
longer holds as soon as the dimension is larger than $2$. 
For a counterexample, consider the
case of non-degenerate centered Gaussian measures $\gamma_1, \ldots, \gamma_J$ on $\R^n$, defined by their 
covariances matrices $S_1, \ldots S_J \in \mathcal{S}_n^{++}$.

According e.g. to \cite{mccann1997convexity}, Example 1.7, the optimal transport map from
$\mathcal{N}(0, S)$ to $\mathcal{N}(0, T)$ is given by

\begin{equation*}
 x \mapsto T^{1/2} ( T^{1/2} S T^{1/2})^{-1/2} T^{1/2} x.
\end{equation*}

From this result, it is possible to give an explicit expression of the iterated barycenter.

On the other hand, according to Theorem 6.1 in \cite{agueh2010barycenters}, the barycenter 
of the $\mu_j$ with weights $1/J$
is the Gaussian measure with covariance matrix the unique positive definite solution of the fixed point equation

\begin{equation*}
 M = \frac{1}{J} \sum_{j = 1}^J \left( M^{1/2} S_j M^{1/2} \right)^{1/2}.
\end{equation*}

One may check that these two covariance matrices do not match in general.

\end{enumerate}
\end{remark}

\subsection{Template Estimation from admissible deformations}

Thanks to Theorem \ref{thm_averaging_deformations_in_admissible_case}, we can study 
the asymptotic behaviour of the barycenter when the number of replications of the warped distributions $J$ increases. Actually, we prove that the barycenter is  an estimator of the template distribution.

 Let $T$ be a process with values in some admissible family of deformations acting on a subset $\mathcal{I} \subset \mathbb{R}^d$.
\begin{equation*}
\begin{array}{rrcl}
T:&\Omega&\rightarrow&\mathcal{T}\left(\mathcal{I}\right)\\
&w&\mapsto&T(w,\cdot),
\end{array}
\end{equation*}
where $\left(\Omega,\mathcal{A},\mathbb{P}\right)$ is an unknown probability space,
   Assume that $T$ is  bounded and has a finite moment $\varphi(.)=\mathbb{E}(T(.))$. Let $T_j$ for $j=1,\dots,J$ be a random sample of realizations of the process $T$. Then, we observe measures $\mu_j$ which are warped by  $T_j$ in the sense that for all , $\mu_j={T_j}_\# \mu$.

\begin{theorem} \label{th:main} Assume that $\mu$ is compactly supported. As soon as $\varphi = {\rm id}$, $\mu_B$  the barycenter of the $\mu_j$'s with weights $1/J$ is a consistent estimate of $\mu$  when $J$ tends to infinity in the sense that a.s
$$  W_2^2( \mu_B, \mu) \stackrel{J \to \infty}{\longrightarrow} 0.$$

Moreover, assuming that $\|T - id\|_{L^2} \leq M$ a.s., we get the following error bound :

\begin{equation*}
 \pr( W_2( \mu_B, \mu) \geq \varepsilon) \leq 2 \exp{ \left( - J \frac{\varepsilon^2}{M^2(1 + c \varepsilon / M)} \right)}.
\end{equation*}

\end{theorem}
Note that when the warping process is not centered, the problem of estimating the original measure $\mu$ is not identifiable and we can only estimate by the barycenter $\mu_B$ the original measure transported by the mean of the deformation process, namely $\varphi_\# \mu$. \vskip .1in
The proof of this theorem relies on the following proposition.

\begin{prop}\label{cor:control}

Let $(T_i)_{i \in I}$ be an admissible family of deformations on a domain $\Omega \subset \R^n$, 
and let $\mu \in \Pro_2( \Omega)$, absolutely continuous with respect to the $n$-dimensional Lebesgue measure. Let $\mu_j = (T_j)_\# \mu$. Denote by $\mu_B$ the
barycenter with equal weights $1/J$. For every $\nu$ in $\Pro_2(\R^d)$, we have

\begin{equation*}
W_2( \mu_B, \nu) \leq \| \frac{1}{J}\sum_{j = 1}^J T_j - T_\nu \|_{L^2( \mu)}
\end{equation*}

where $T_\nu$ is the Brenier map from $\mu$ to $\nu$.

\end{prop}

\begin{proof}

With the explicit expression of the barycenter, we know that the Brenier map from $\mu$ to $\mu_B$ 
is $1/J \sum_{j = 1}^J T_j$, which implies that

\begin{equation*}
 \pi = ( \frac{1}{J}\sum_{j = 1}^J T_j, T_\nu )_\# \mu
\end{equation*}

is a coupling of $\mu_B$ and $\nu$. Consequently,

\begin{equation*}
 W_2^2( \mu_B, \nu) \leq \int |\frac{1}{J}\sum_{j = 1}^J T_j(x) - T_\nu(x)|^2 \mu(d x).
\end{equation*}

\end{proof}

\section{Statistical Applications} \label{s:appli}
\subsection{Distribution Template estimation from empirical observations}
In many situations, the issue of estimating the mean behaviour of random observations plays a crucial role to analyze the data, in image analysis, kinetics in biology for instance. For this, we propose to use the iterative barycenter of  the empirical distribution as a good estimate of the {\it mean} information conveyed by the data. Moreover, this estimate has the advantage that if the different distribution  are warped from an unknown distribution, the empirical iterative barycenter converges to this pattern when the number of replications grows large.
 \vskip .1in
Assume we observe $j=1,\dots,J$ samples  of   $i=1,\dots,n$ points  $X_{i,j} \in \mathbb{R}^d$ which are i.i.d realizations of measures $\mu_j$.   Hence we observe cloud points or in an equivalent way ${\mu}_{j}^{n}=\frac 1 n \sum_{i=1}^n \delta_{X_{i,j}}$ empirical versions of the measures  $\mu_j$. It is well-known that considering the mean with respect to the number of samples $J$ of all observation points does not provide a good model of the {\it mean} behaviour. Instead  we here consider the iterative barycenter $\mu^J_B=IB(\mu_j,\frac 1 J)$  defined in Definition~\ref{def:IB}.
  The following proposition shows that the  barycenter of the  empirical distributions provides a good estimate for this {\it mean shape}. We point out that this estimator corresponds to the so-called Fr\'echet mean of the empirical measures.
\begin{prop} \label{prop:lafin}
Assume that the observations $X_{.,j} \sim \mu_j$ are warped by a centered admissible deformation process from an unknown template distribution $\mu$ continuous with respect to Lebesgue measure. Set ${\mu^{n,J}_B}  \in \text{Bar}({\mu}_{j}^{n}, \frac{1}{J})$, an empirical mean of the empirical distribution.
As $n \rightarrow + \infty$, we have
\begin{equation*}
 {\mu_B^{n,J}} \longrightarrow \mu_B^J.
\end{equation*}
Moreover, when $n \rightarrow + \infty$ and $J\rightarrow + \infty$, ${\mu_B^{n,J}}$ is a consistent estimate of $\mu$, in the sense that  $$ {\mu_B^{n,J}} \longrightarrow \mu \quad {\rm in} \: W_2 \: \: {\rm distance}.$$
\end{prop}
We point out that   $\mu_B^{n,J}$ exits but is not unique. Actually, to ensure uniqueness, one may consider a regularized version of the empirical measures. For instance let  $\gamma_\varepsilon$ denotes a $\mathcal{N}(0, \varepsilon {\rm I}_d)$ measure. Set 
$\widehat{\mu_{j}^n} = {\mu}_j^n * \gamma_{1/n}.$ In this case
 $\widehat{\mu^{n,J}_B} = \text{Bar}({\mu_{j}^n}, \frac{1}{J})$ is uniquely defined and 
as $n \rightarrow + \infty$, we have
$\widehat{\mu_B^{n,J}} \longrightarrow \mu_B$ in Wasserstein distance.  Note that any other regularization scheme may be used as soon as the corresponding measures converge to the true measures in Wasserstein distance when $n$ goes to infinity.  \vskip .1in
An  important application is given by the issue of ensuring equality between the candidates in an exam with several different referees. This constitutes a natural extension  of the work in~ \cite{ dupuy2011non}  to higher dimensions. \\
\indent Consider an examination with a large number of candidates, such that it is impossible to evaluate the candidates one after another. The students are divided into $J$ groups, and $J$ boards of examiners are charged to grade these  groups: each board grades one group of candidates. The evaluation is performed by assigning $p$ scores.
The $J$ different boards of examiners are supposed to behave the same way, so as to respect the equality among the candidates. Moreover it is assumed that the sampling of the candidates is perfect in the sense that it is done in such a way that each board of examiners evaluates candidates with the same global level. Hence, if all the examiners had the same requirement levels, the distribution of the ranks would be the same for all the boards of examiners. Here, we aim at balancing the effects of the differences between the examiners, gaining equity for the candidates.
The situation can be modeled as follows. For each group $j$ among $ J$ groups of candidates, let ${\bf X}^j= \left\{X^j_i \in \mathbb{R}^p, \: i=1, \dots, n\right\}$ denote the scores of the students within this group. Let $\mu_j$ and $\mu_{j}^{n}$ be respectively  the measure and the empirical measure of the scores in the $j$-th group.\vskip .1in
 We aim at finding the average way of ranking, with respect to the ranks that were given within the $p$ bunches of candidates. For this, assume that there is such an average measure, and that each group-specific measure is warped from this reference measure  by a random process. A good choice is given by the barycenter measure 
In order to obtain a global common ranking for the $N$ candidates, one can now replace the $p$ group-specific rankings by the sole ranking based on barycenter measure. Indeed each measure can be pushed towards the barycenter. As a result, we obtain a new set of scores for the $N$ candidates, which can be interpreted as the scores that would have been obtained, had the candidates been judged by an average board of examiners.
\subsection{Principal Component  Analysis with Wasserstein distance}
   Once we have succeeded in defining a mean of a collection of distributions, 
then the second step consists in understanding the variability of the  the different experiments 
with respect to this {\it average} distribution, which is, in statistics,  the aim of the so-called PCA analysis. In a Euclidean space, a natural way to define principal components is through the minimization of the variance of the residuals.  This concept has been extended to non Euclidean situations  such as manifolds, Kendall's
shape spaces in~\cite{MR2168993}. The principal component directions  are replaced by  principal 
component \emph{curves} from a suitable family of curves, e.g. geodesics. In our framework, we generalize this idea to the Wasserstein distance.\vskip .1in

As previously, let $\mu_1,\dots,\mu_J$ be measures in $ \Pro_2(\R^d)$ and let $\mu_B$ be the mean defined as the Barycenter $\mu_B=Bar(\mu_j,\frac 1 J)$. Let $S_j,\: j=1,\dots,J$ be the transport plan 
between the $\mu_j$'s and $\mu_B$ and write $\mu_j={S_j}_\# \mu_B$. Assume the $(Id,S_j)$ are an admissible family of transformations. Clustering the experiments
 in order to build coherent groups is usually  achieved by comparing a distance between 
these distributions. Here by choosing the Wasserstein distance we get that 
$$ W_2^2(\mu_B,\mu_j)= \int |S_j(x)-x|^2 d\mu_B = \| S_j-{\rm id}\|^2_{L^2(\mu_B)}.$$
Hence statistical analysis of the distributions $\mu_j$'s amounts to clustering their 
Wasserstein square distance $\| S_j-{\rm id}\|^2_{L^2(\mu_B)} \in \mathbb{R}^+$.\vskip .1in

It is known that the Wasserstein metric endows the space of probability measures with a formal
Riemannian structure, in which it is possible to define geodesics, tangent spaces, etc., see \cite{MR2640651}, \cite{ambrosio2004gradient}.
We propose here a method of principal
component analysis using Wasserstein distance based on geodesics of
the intrinsic metric, which follows the ideas developed in \cite{MR2640651}. 
For this,  consider a geodesic segment $\gamma$ at point $\mu$ with direction $T$,  which can be written as 
$$ \forall t \in [0,1],\: \gamma(t)=((1-t) Id+ t T)_\# \mu.$$ 
We extend the definition of  $\gamma$ to every $t \in \R$, with the important provision that
$\gamma$ is in general \emph{not} a geodesic curve for the whole range of $t \in \R$.
We perform PCA with respect to this family of curves which we somewhat abusively refer to as ``geodesic curves``
on their extended range.
We will come back to this discussion at the end of our analysis.

For every $\mu_j$, the natural
 distance to the geodesic curve $\gamma$ is given by
$$d^2(\mu_j,\gamma)=\inf_{t \in \R }  W_2^2(\mu_j,\gamma(t)).$$
\begin{definition}
A geodesic $\gamma_1$ is called a {\it first generalized principal component geodesic (GPCG)} to the $\mu_j$'s if it minimizes the following quantity
\begin{equation} \label{geoPCA}
\gamma \mapsto  \frac{1}{J} \sum_{j=1}^J d^2(\mu_j,\gamma)
\end{equation}
Then we define the {\it second GPCG}, a geodesic $\gamma_2$  which minimizes \eqref{geoPCA} over all geodesics that have at least 
one point in common with $\gamma_1$ and that are orthogonal to $\gamma_1$ at all points in common. \\
Every point $\mu^\star$ that minimizes $\mu \mapsto \frac{1}{J} \sum_{j=1}^J W_2^2(\mu_j,\mu)$ over 
all common points  of $\gamma_1$ and $\gamma_2$ will be called a {\it principal component geodesic mean}. 
Given the first and the second principal component geodesics $\gamma_1$ and $\gamma_2$  with principal 
component geodesic mean $\mu^\star$ we say that a geodesic $\gamma_3$ is a {\it third principal component 
geodesic} if it minimizes \eqref{geoPCA} over all geodesics that meet previous principal components 
orthogonally at $\mu^\star$. Analogously, principal component geodesics of higher order
are defined. \end{definition}
Here we will focus on the computation of the first geodesic component $\gamma_1$.  We will only consider geodesic curves from $\mu_B$, in that case note that $\mu_B=\mu^\star$. Hence we root our analysis at the central point given by $\mu_B$ which plays the role of the mean of the sample of distributions. In this setting,
we first define a geodesic starting at a measure  $\mu_B$  directed by a map $T$ as 
$\gamma(t)=((1-t)Id+t T)_\# \mu_B$. We consider a family  of maps $T$ such that $(Id,S_j,T)$ is an admissible family of deformations. Hence, in that case, the distance of any measure $\mu_j$ 
with respect to such a geodesic can be written as
\begin{align*}
d^2(\mu_j,\gamma) & = \inf_{t \in \R }  \int \left[((1-t) Id+ t T) \circ S_j^{-1}-Id   \right]^2 d\mu_j(x) \\
& = \inf_{t \in \R }  \int \left[((1-t) Id+ t T) -S_j   \right]^2 d\mu_B(x) \\
& =  - \frac{<S_j  -Id ,T-Id>_{L^2(\mu_B)}^2}{\|T-Id \|^2_{L^2(\mu_B)}} 
\end{align*}
where $\|.\|_{L^2(\mu_B)}$ denotes the quadratic norm with respect to the measure $\mu_B$ with corresponding scalar product $<.,.>_{L^2(\mu_B)}$. 
Finally PCA with respect to Wasserstein distance amounts to minimizing with respect to $T$ 
 the  quantity $\sum_{j=1}^J d^2(\mu_j,\gamma)$, which can be written as 
$$ T \mapsto  - \sum_{j=1}^J |<S_j  -Id,\frac{ (T-Id) }{ \| T-Id \|_{L^2(\mu_B)}}>|^2_{L^2(\mu_B)} .$$ 
If we set $v= T-Id$, this maximization can be written as finding the solution to 
$$ {\rm arg}\max_{v,\: \| v\|_{L^2(\mu_B)}=1} \sum_{j=1}^J | < S_j -  Id,v>|^2_{L^2(\mu_B)},$$ 
which corresponds 
 to the functional principal component analysis of the maps $S_j ,\: j=1,\dots,J$
  in the space $L^2(\mu_B)$.  This analysis can be achieved using tools defined for instance 
in \cite{MR2168993}. Finally, if we get $T^{(1)}$ the map  corresponding to the first functional principal component, the corresponding principal geodesic if obtained by setting $\gamma^{(1)}(t)=((1-t) Id+ t T^{(1)})_\# \mu_B$. The other principal components can be computed using the same procedure.\\
In the one dimensional case, the situation is  simpler since,  the distance between $\mu_j$ with distribution function $F_j$ and a geodesic $\gamma$ from $\mu_B$ with distribution function $F_B$ to $T_\# \mu_B$ is given by
$$d^2(\mu_j,\gamma)= \inf_{t \in \R }  \int \left[((1-t) Id+ t T) \circ F_B^{-1}- F_j^{-1}  \right]^2 dt. $$
 Hence PCA analysis amounts to maximizing for all functions $T$
\[
T\mapsto  \sum_{j=1}^J |<S_j  \circ F_B^{-1}-F_B^{-1},\frac{ (T-Id)\circ F_B^{-1} }{ \| (T-Id)\circ F_B^{-1} \|}>|^2,\]
which corresponds to the functional PCA  of the maps $S_j ,\: j=1,\dots,J$
  in the space $L^2(\mu_B)$ without any restriction. \vskip .1in

Let us come back to the caveat that the curves chosen are not Wasserstein geodesics on the entire parameter range.
It is easy to check in the one dimensional case (see \cite{ambrosio2004gradient}) that a curve $\gamma(t) = ((1-t) Id+ t T)_\# \mu$ is a geodesic curve
for all $t \in \R$ such that $(1-t) \text{Id} + t T$ is an increasing function. Assuming $T'$ takes values in the interval $[a, b]$,
$0 < a < 1 < b$, this means that $\gamma$ is a geodesic curve for all $t \in [1/(a-1), 1/(b-1)]$.
Once the analysis above yields the expression of $T^{(1)}$ and the $t^*_j$ minimizing $d^2(\mu_j, \gamma)$, it is possible to check whether they
fall in this range. Actually, 
$$t^*_j = \frac{<S_j-Id,T^{(1)}-Id>_{L^2(\mu_B)}}{\|T^{(1)}-Id \|_{L^2(\mu_B)}^2} .$$ Hence, when the measures $\mu_j$ are not too far from their barycenter
(i.e. when the $\|S_j - \text{Id} \|_\infty$ are small) these conditions are met.\vskip .1in
Within this framework, we can analyze the toy example of translation effect, studied in \cite{ Gamboa-Loubes-Maza-07} or in \cite{ GALLON:2011:HAL-00593476:1}. Here consider i.i.d random variables $X_i \in \mathbb{R}^p,\: i=1,\dots,n$ who are translated by parameters $\theta_j=(\theta_j^1,\dots,\theta_j^p),\: j=1,\dots,J$. Hence the observation model is $X_{ij}=X_i+\theta_j$. Let $\mu$ be the distribution of the $X_i$'s and assume that this distribution admits a density $f$ with respect to the Lebesgue measure. Hence $\mu_j$'s, the distributions of the $X_{ij}$ are given by $$ \mu_j={T_j}_\#\mu $$ with $T_j(x)=x+\theta_j$. They admit densities with respect to Lebesgue measure, $f_j$'s which are such that $\forall x \in \mathbb{R}^p,\: f_j(x)=f(x-\theta_j)$.  In this case, $\mu_B$ the Barycenter of the $\mu_j$'s exists and is characterized by its density $f_B(x)=f(x-\bar{\theta})$, with $\bar
{\theta}=\frac{1}{J} \sum_{j=1}^J \theta_j$. 
 \vskip .1in
In this context, each distribution $\mu_j$ can be expressed as $$ \mu_j={S_j}\# \mu_B, \quad S_j(x)=\frac{1}{J}\sum_{k=1}^J T_k^{-1}\circ T_j(x)=x+\bar{\theta}-\theta_j. $$
Now finding the first geodesic component rooted in $\mu_B$ amounts to maximize with respect to $T$ the quantity
 \begin{align*} T \mapsto & \sum_{j=1}^J  |<S_j-Id ,\frac{ (T-Id) }{ \| T-Id \|_{L^2(\mu_B)}}>_{L^2(\mu_B)}|^2 \\
 & = \sum_{j=1}^J \|\theta_j-\bar{\theta}\|^2  \frac{ (\int ( T-Id ) d\mu_B)^2 }{ \| T-Id \|^2_{L^2_{(\mu_B)}}}
  \end{align*} which is achieved by choosing $T=Id + c $ for all constant $c$, where $\|.\|$ denotes the norm in $\mathbb{R}^p$. Hence the first principal geodesic component is given by $\mu^1_t$ with density $f(x-\bar{\theta}-t)$ while the variance explained is given by $\sum_{j=1}^J \|\theta_j - \bar{\theta}\|^2$. This corresponds actually to the variance of the deformations.

\section{Appendix} \label{s:append}
Proof of Lemma~\ref{lem:iff}
\begin{proof}

The existence of a solution of the multimarginal problem (\ref{eq:a_minimiser}) follows from a classical compactness argument.

Let $\gamma \in \Gamma(\mu_1,...,\mu_J)$ and set $\nu=T_\# \gamma$. For all $1 \leq j \leq J$,
\[
W_2^2(\nu,\mu_j) \leq \int \|T(x_1,...,x_J) - x_j\|^2 d\gamma(x_1,...,x_J),
\]
and thus,
\begin{equation}\label{eq:premier_sens}
\sum_{1 \leq j \leq J} \lambda_j W_2^2(\nu,\mu_j) \leq \int \sum_{1 \leq j \leq J} \lambda_j  \|T(x_1,...,x_J) - x_j\|^2 d\gamma(x_1,...,x_J).
\end{equation}

For $1 \leq j \leq J$ and a probability measure $\hat{\nu}$, denote $\hat{\pi_j}$ a minimiser of
\[
\int \|x-x_j\|^2d\pi(x,x_j)
\]
over all $\pi \in \Gamma(\nu,\mu_j)$ and define $\Pi$ by
\begin{equation}\label{eq:def_Pi}
\Pi(A\times B_1 \times ... \times B_J) = \hat{\nu}(A) \frac{\pi_1(A\times B_1)}{\hat{\nu}(A)}... \frac{\pi_J(A\times B_J)}{\hat{\nu}(A)}
\end{equation}

Suppose now that $\gamma$ is moreover a minimizer of (\ref{eq:a_minimiser}), we want to show that $\nu=T_\# \gamma$ is a barycenter. Indeed,
\begin{align}
\sum_{1 \leq j \leq J} \lambda_j W_2^2(\hat{\nu},\mu_j) = & \sum_{1 \leq j \leq J} \lambda_j \int \|x - x_j\|^2 d\Pi(x,x_1,...,x_J) \label{eq:par_def}\\
= & \int \sum_{1 \leq j \leq J} \lambda_j \|x - x_j\|^2 d\Pi(x,x_1,...,x_J) \nonumber\\
\geq & \int \inf_{z \in E} \sum_{1 \leq j \leq J} \lambda_j \|z - x_j\|^2 d\Pi(x,x_1,...,x_J) \label{eq:inegalite}\\
= & \int \sum_{1 \leq j \leq J} \lambda_j \|T(x_1,...,x_J) - x_j\|^2 d\Pi(x,x_1,...,x_J) \label{eq:T_minimal}\\
\geq &\int \sum_{1 \leq j \leq J} \lambda_j \|T(x_1,...,x_J) - x_j\|^2 d\gamma(x_1,...,x_J) \label{eq:gamma_minimal}\\
\geq & \sum_{1 \leq j \leq J} \lambda_j W_2^2(\nu,\mu_j)\label{eq:par_ps},
\end{align}
where (\ref{eq:par_def}) holds by definition (\ref{eq:def_Pi}) and (\ref{eq:T_minimal}) holds since for fixed $x_1,...,x_J$, the sum $\sum_{1 \leq j\leq J} \lambda_j \|x-x_j\|^2$ attains its minimum at $x=T(x_1,...,x_J)=\sum_{1 \leq j\leq J} \lambda_j x_j$. The inequality (\ref{eq:gamma_minimal}) holds since $\gamma$ is optimal and (\ref{eq:par_ps}) holds by (\ref{eq:premier_sens}).

Since $\hat{\nu}$ was arbitrary, this shows that $\nu$ is a barycenter.

On the other hand, taking $\hat{\nu}$ a barycenter, inequality (\ref{eq:inegalite}) becomes an equality, so that, for $\Pi$-almost all $(x,x_1,...,x_J)\in \R^{d\times (J+1)}$,
\[
\sum_{1 \leq j \leq J} \lambda_j \|x - x_j\|^2 = \inf_{z \in E} \sum_{1 \leq j \leq J} \lambda_j \|z - x_j\|^2
\]
which shows that $x=T(x_1,...,x_J)$ $\Pi$-almost surely, and thus that $\hat{\nu}=T_\#\Pi_p$, where $\Pi_p$ is the projection of $\Pi$ over the last $J$ marginals. The fact that $\Pi_p$ is a solution of (\ref{eq:a_minimiser}) is a consequence of (\ref{eq:premier_sens}) and equality (\ref{eq:gamma_minimal}).

\end{proof}

Proof of Theorem~\ref{th:EmpiricVersion}
We know by Lemma~\ref{lem:iff} that for all $n\geq 1$, there exists $\gamma^n\in \Gamma(\mu_1,...,\mu_J)$ such that $\mu^n=T_\#\gamma^n$.
We first show that the sequence $(\gamma^n)_{n\geq 1}$ is tight. Let $B_1, \ldots, B_J$ be large balls in $\R^d$, 
we have

\begin{align*}
\gamma^n((B_1 \times \ldots B_J)^c) & = \gamma^n( \cup_{j = 1}^J E \times \ldots \times E \times B_j^c \times E \ldots \times E) \\
{} & \leq \sum_{j = 1}^J \gamma_n (E \times \ldots \times E \times B_j^c \times E \ldots \times E) \\
{} & = \sum_{j = 1}^J \mu_j^n(B_j^c).
\end{align*}

Thus, tightness of the sequences $(\mu_j^n)_{n\geq 1}$ guarantees tightness of $(\gamma^n)_{n \geq 1}$. Note that the under the assumption of the convergence of $\mu_j^n$, $n \geq 1$ in
Wasserstein distance, we  recover the compactness of $(\gamma^n)_{n \geq 1}$ in Wasserstein topology. Indeed, denote $\gamma$ any weak limit of the tight sequence $(\gamma^n)_{n \geq 1}$, the second moments are converging:
\begin{align*}
 \int |x|^2 d \gamma^n & = \sum_{j = 1}^J \int |x_j|^2 d \mu_j^n \\
{} & \rightarrow \int |x_j|^2 d \mu_j= \int |x|^2 d \gamma.
\end{align*}
Here we used the fact that Wasserstein's convergence coincides with weak convergence together with the convergence of second order moments. 
The above implies tightness of the sequence of barycenters ${\mu}^n_B, n\geq 1$ : indeed, it is the push-forward of the tight sequence 
$(\gamma^n)_{n \geq 1}$ by the application $T : \R^{d \times J} \rightarrow \R^d$, which is 
 Lipschitz continuous (with Lipschitz constant bounded by $1$).
It is readily checked that this operation preserves tightness, as it preserves convergence 
(in weak and Wasserstein topologies).

We conclude by showing that any limiting point ${\mu}^\infty$ is a minimizer for the barycenter problem
associated with $\mu_1, \ldots, \mu_J$. Denote by ${\mu}_B$ a barycenter of $\mu_1, \ldots, \mu_J$.
Since ${\mu}^n_B$ is a barycenter for $\mu_1^n, \ldots, \mu_J^n$, we have

\begin{equation*}
 \sum_{j = 1}^J \lambda_j W_2^2({\mu}^n_B, \mu_j^n) \leq \sum_{j = 1}^J \lambda_j W_2^2( \mu_B, \mu_j^n).
\end{equation*}

Since, up to a subsequence, ${\mu}^n_B \rightarrow {\mu}^\infty$ in Wasserstein distance, letting $n \rightarrow + \infty$ shows
\begin{equation} \label{titi}
  \sum_{j = 1}^J \lambda_j W_2^2({\mu}^\infty, \mu_j) \leq \lim \sum_{j = 1}^J \lambda_j W_2^2( \mu_B, \mu_j^n).
\end{equation}

%
%


Proof of Theorem~\ref{thm_averaging_deformations_in_admissible_case}
\begin{proof}
For the first part, \eqref{firstpart}, we use induction on $J$. For $J = 1$, the result is obvious. Suppose then that it is established for $J \geq 1$.
Choose $T_1, \ldots, T_{J + 1}$ from a family of admissible deformations, and fix $\lambda_1, \ldots, \lambda_{J+1}$ with
$\sum_{j = 1}^{J + 1} \lambda_j = 1$. 
Using the definition of the iterated barycenter, we have

\begin{align*}
  IB( ( \mu_j , \lambda_j )_{1 \leq j \leq J + 1} ) & = \text{Bar} \left(  IB \left( ( \mu_j , \lambda_j )_{1 \leq j \leq J} ),  \sum_{j = 1}^{J} \lambda,_j \right), (\mu_{J + 1}, \lambda_{J+1}) \right) \\
{} & = \text{Bar} \left( \left( (\frac{1}{ \Lambda_j} \sum_{j = 1}^J \lambda_j T_j )_\# \mu, \Lambda_J \right), (\mu_{J+1}, \lambda_{J+1})  \right)
\end{align*}

where we set $\Lambda_J = \sum_{j = 1}^J \lambda_j$.

Set $\nu = (\frac{1}{ \Lambda_j} \sum_{j = 1}^J \lambda_j T_j )_\# \mu$. As $\mu_{J+1} = {T_{J+1}}_\# \mu$, we have also
$\mu = (T_{J+1})^{-1}_\# \mu_{J+1}$, and

\begin{align*}
 \nu &  = (\frac{1}{ \Lambda_j} \sum_{j = 1}^J \lambda_j T_j ) \circ  (T_{J+1})^{-1} _\# \mu \\
{} & = (\frac{1}{ \Lambda_j} \sum_{j = 1}^J \lambda_j T_j \circ  (T_{J+1})^{-1} ) _\# \mu_{J+1}.
\end{align*}

Now, observe that by assumption all the maps $ T_j \circ  (T_{J+1})^{-1}$ are gradients of convex functions, so that their convex combination
also is. By Brenier's theorem, the map

\begin{equation*}
 \mathcal{T} = \frac{1}{ \Lambda_j} \sum_{j = 1}^J \lambda_j T_j \circ  (T_{J+1})^{-1}
\end{equation*}

is the Brenier map from $\mu_{J +1}$ to $\nu$. We deduce that the barycenter of $\nu$ and $\mu_{J+1}$ is 

\begin{align*}
 {} & \left( \lambda_{J+1} \text{Id} + \Lambda_J \mathcal{T} \right)_\# \mu_{J+1} \\
{} & = \left( \lambda_{J+1} T_{J+1} + \Lambda_J \mathcal{T} \circ T_{J+1}  \right)_\# \mu \\
{} & = ( \sum_{j = 1}^{J + 1} \lambda_j T_j )_\# \mu.
\end{align*}

This finishes the first part of the proof.\vskip .1in

For the identification of the barycenter and the iterative barycenter  given in~\eqref{prop:clememe}, we proceed as follows.
Set $T(x_1, \ldots, x_J) = \sum_{j = 1}^J \lambda_j x_j$ for $x_1, \ldots, x_J \in \R^d$.
Proposition 4.2 of \cite{agueh2010barycenters} claims that the barycenter of
$( \mu_j , \lambda_j )_{1 \leq j \leq J}$, denoted by $\mu_B$, satisfies $\mu_B = T_\# \gamma$
where $\gamma \in \Pro((R^d)^J)$ is the unique solution of the optimization problem

\begin{equation*}
\inf \left\{ \int \sum_{j = 1}^J \lambda_j |T(x) - x_j|^2 d \gamma(x_1, \ldots, x_J), \quad \gamma \in \Pi(\mu_1, \ldots, \mu_J)  \right\}
\end{equation*}

where $\Pi( \mu_1, \ldots, \mu_J)$ is the set of probability measures on $\R^{dJ}$ with $j$-th marginal $\mu_j$, $1 \leq j \leq J$.
This can be rewritten as

\begin{equation*}
\frac{1}{2} \inf \left\{ \int \sum_{i, j = 1}^J \lambda_i \lambda_j |x_i - x_j|^2 d \gamma(x_1, \ldots, x_J), \quad \gamma \in \Pi(\mu_1, \ldots, \mu_J)  \right\}.
\end{equation*}

The integral is bounded below by $\sum_{i, j = 1}^J \lambda_i \lambda_j W_2^2(\mu_i, \mu_j)$ (because each term of the sum is bounded by
$W_2^2(\mu_i, \mu_j)$). On the other hand, choosing

\begin{equation*}
 \gamma = (T_1, \ldots, T_j)_\# \mu,
\end{equation*}

we see that $\gamma \in \Pi(\mu_1, \ldots, \mu_J)$, and that

\begin{equation*}
 \int |x_j - x_i|^2 d \gamma = \int |T_j(x) - T_i(x)|^2 d \mu(x) = \int [T_j \circ{T_i}^{-1}(x) - x|^2 \mu_i(d x) = W_2^2( \mu_i, \mu_j).
\end{equation*}

Thus $\gamma$ is optimal, and we have 

\begin{equation*}
 \mu_B = T_\# \gamma = ( \sum_{j = 1}^J \lambda_j T_j )_\# \mu.
\end{equation*}

\end{proof}

Proof of Theorem~\ref{th:main}
\begin{proof}
Using the results of Proposition~\ref{cor:control}, we get that
$$
 W_2^2( \mu_B, \mu) \leq \int |\frac{1}{J}\sum_{j = 1}^J T_j(x) - x|^2 \mu(d x).$$
Almost sure convergence towards $0$ of  $\frac{1}{J}\sum_{j = 1}^J (T_j-{\rm id})$ is 
directly deduced from Corollary 7.10 (p. 189) in \cite{Ledoux91}, which is an extension
 of the Strong Law of Large Numbers to Banach spaces. Then the result follows from dominated convergence.

Likewise, obtaining error bounds is straightforward.
Assuming that $\|T - id \|_{L^2} \leq M$ a.s., we can use Yurinskii's version of Bernstein's inequality 
in Hilbert spaces 
(\cite{yurinski1976exponential}, p. 491) to get the result announced.
\end{proof}
{\bf Acknowledgements: } We thank an anonymous referee for his/her comments and suggestions which contribute to numerous improvements in the paper.
\bibliography{WarpWasser2}{}
\bibliographystyle{plain}

\end{document}